\documentclass[10pt,letterpaper]{amsart}
\usepackage[utf8]{inputenc}
\usepackage[english]{babel}
\usepackage{subfigure}
\usepackage{amsmath}
\usepackage{amsthm} %Proof
\usepackage{hyperref} %Referencias
\usepackage{amsfonts}
\usepackage{tikz-cd}
\usepackage{amssymb}
\usepackage{csquotes}
\usepackage[
backend=biber,
style=alphabetic,
sorting=nty,doi=false,isbn=false, url=false
]{biblatex}
\usepackage{accents}
\newlength{\dhatheight}
\newcommand{\doublehat}[1]{%
    \settoheight{\dhatheight}{\ensuremath{\hat{#1}}}%
    \addtolength{\dhatheight}{-0.15ex}%
    \widehat{\vphantom{\rule{-2pt}{\dhatheight}}%
    \smash{\widehat{#1}}}}

\newtheorem*{theorem*}{Theorem}
\newtheorem{theorem}{Theorem}[section]
\newtheorem{prop}[theorem]{Proposition}%[subsection]
\newtheorem{defi}[theorem]{Definition}%[subsection]
\newtheorem{lemma}[theorem]{Lemma}%[subsection]
\newtheorem{corollary}[theorem]{Corollary}%[subsection]
\newtheorem*{corollary*}{Corollary}
\newtheorem*{prop*}{Proposition}

\newcommand{\C}{\mathbb{C}}
\newcommand{\Z}{\mathbb{Z}}

\newcommand{\Q}{\mathbb{Q}}
\newcommand{\N}{\mathbb{N}}

\newcommand{\et}{\text{\'et}}

\newcommand{\spc}{\text{Spec}}

\newcommand{\ho}{\text{Hom}}

\newcommand{\CH}{\text{CH}}

\theoremstyle{remark}
\newtheorem{remark}[theorem]{Remark}
\tikzcdset{
  cells={font=\everymath\expandafter{\the\everymath\displaystyle}},
}
\newcommand{\info}{{% additional braces for segregating \footnotesize
  \bigskip
  \footnotesize

  \textsc{ENS de Lyon, UMPA, UMR 5669, 46 all\'ee d'Italie, 69364 Lyon Cedex 07, France}\par\nopagebreak
  \textit{E-mail address}: \texttt{ivan.rosas\_soto@ens-lyon.fr}
  }}
  
  \subjclass[2010]{14C25, 14F20, 19E15}
\keywords{Algebraic cycles, \'etale motives, étale cohomology, motivic cohomology}

\usepackage[left=3cm,right=3.5cm,top=2cm,bottom=3.2cm]{geometry}
\author{Iv\'an Rosas-Soto}
\date{October 2024}
\title{Fourier transform for étale motivic cohomology}
\begin{document}

\maketitle
\begin{abstract}
In the present article, we study the integral aspects of the Fourier transform of an abelian variety $A$ over a field $k$, using étale motivic cohomology, following the ideas and theory given by Moonen, Polishchuk and later by Beckman and de Gaay Fortman. We prove that there exists a PD-structure over the positive degree part of the étale Chow ring $\CH^\et_{>0}(A)$ with respect to the Pontryagin product.
\end{abstract}

\section{Introduction}

Consider an abelian variety $A$ defined over a field $k$ and let $\CH^*(X)$ and $\CH_\et^*(X)$ be the Chow and étale Chow rings, respectively, with respect to the intersection product. In 2004, Ésnault gave a counterexample for the existence of a PD structure (after a question by Kahn) in $\CH^*(X)$ when $k=\mathbb{F}_q$ is a finite field, and this negative result also extends to étale motivic cohomology. Although this is the counterexample of Ésnault, there is still hope to get a PD-structure in other generalities, for example when the base field is algebraically closed. Later, in 2010, Moonen and Polishchuk proved the existence of a PD structure in the ideal $\CH_{>0}(X)\subset \CH_*(X)$ with respect to the Pontryagin product induced in the Chow groups by the multiplication in $A$, and even better, if the basis field is algebraically closed, then the augmented ideal $I=\CH_{>0}(X)\cup \CH_0(X)^{\deg=0}\subset \CH_*(X)$ has a PD-structure. 

For an abelian variety $A$ over a field $k$ and its dual variety $\widehat{A}$, the correspondence defined in $\CH^*(A\times \widehat{A})_\Q$ by the Poincaré bundle $\mathcal{P}_{A}$ induces an equivalence of their derived categories, Chow groups with rational coefficients, and cohomology groups, and it is called the Fourier transform of abelian varieties, denoted by $\mathcal{F}_A$. The existence of the Fourier transform has deep consequences in the motivic behavior of the abelian variety $A$, for example, Laumon in \cite{Lau} used it to simplify Deligne's proof of Weil's conjectures, or Shermenev, Deninger and Murre, in \cite{She} and \cite{DM} respectively, to obtain the Chow-Künneth decomposition of the motive associated to an abelian variety $A$. For the case where $k$ is a separably closed field, the Fourier transform preserves integral $\ell$-adic classes (or integral classes in Betti cohomology when $k=\C$); the question is whether this is the case for algebraic cycles. We should note that there is a connection between the intersection product of cycles in an abelian variety $A$ and the Pontryagin product, which is given by the Fourier transform, by exchanging one for the other as follows
\begin{align*}
    \mathcal{F}_A(x * y) = \mathcal{F}_A(x) \cdot \mathcal{F}_A(y), \quad \mathcal{F}_A(x \cdot y) = (-1)^{\text{dim}(A)}\mathcal{F}_A(x) * \mathcal{F}_A(y)
\end{align*}

In \cite{MP}, Moonen and Polishchuk studied the integral aspect of the  Fourier transform, asking whether it lifts to a morphism 
\begin{align*}
    \mathcal{F}_A: \CH^*(A) \to \CH^*(\widehat{A}).
\end{align*}
in a motivic sense, defining the integral Fourier transform up to an integral constant. However, the existence of an integral Fourier transform is false due to the counterexample given by Ésnault in \cite{esn}, if the field is algebraically closed, this question remains open. With the existence of an integral Fourier transform for étale motivic cohomology, we might be able to find explicit projectors for an integral Chow-Künneth decomposition of the étale motive associated to an abelian variety $A$.

In this paper we aim to answer this question, namely whether the existence of a Fourier transform is true, but for étale motivic cohomology over a field $k =\bar{k}$. This allows us to address the construction of projectors (for étale motivic cohomology groups) for the integral decomposition of the étale motive associated to an abelian variety $A$. With this goal in mind, we begin by proving the following result about the PD-structure in the ideal $\CH_{>0}^\et(M) \subset \CH_*^\et(M)$ associated with the étale motivic cohomology groups of a commutative monoid $M$:

\begin{corollary*}[{Corollary \ref{coro1}}]
Let $M$ be a commutative monoid with identity in the category of quasi-projective $k$-schemes, such that the product morphism $\mu : M\times M \to M$ is proper. Let $p_d:\text{Sym}^d(M)\to M$ be the morphism induced by the iterated multiplication map $M^d \to M$. Then the maps $\gamma_d^M: \CH_{>0}^\et(M) \to \CH_*^\et(M)$ defined by $x \mapsto (p_d)_* \gamma_d(x)$ define a PD-structure on the ideal $\CH_{>0}^\et(M)\subset \CH_*^\et(M)$.    
\end{corollary*}

If we consider an abelian variety $A$ over a field $k=\bar{k}$, we have a PD-structure over the augmented ideal $I=\CH_{>0}^\et(X)\cup \CH_0^\et(X)^{\deg=0}\subset \CH_*^\et(X)$ as one has in the case of Chow groups:

\begin{corollary*}[{Corollary \ref{coro2}}]
Let $k = \bar{k}$ be a field and let $A$ be an abelian variety over $k$, then there exists a canonical PD-structure, with respect to the Pontryagin product, on the augmentation ideal in $\CH_*^\et(A)$, generated by $\CH_{>0}^\et(A)$ together with the $0$-cycles of degree 0.
\end{corollary*}

This is the existence of a PD-structure of $\CH_\et(A)$, where $A$ is an abelian variety, and with respect to the Pontryagin product, as it was obtained in \cite{MP} for the case of classical Chow groups. 

Having obtained the analogues of the PD-structure for étale motivic cohomology, we proceed to study the integral properties of the Fourier transform for étale Chow groups. For this we will use the definition in the spirit of the definition given by Beckmann and de Gaay Fortman in \cite{BG}, which is a weaker version of the integral Fourier transform (in the motivic sense) given by Moonen and Polishchuk in \cite{MP}.

\begin{defi}
    Let $A$ be an abelian variety over $k$ and let $\mathcal{F}_\et: \CH_\et(A)\to \CH_\et(\widehat{A})$ be a group homomorphism. We call $\mathcal{F}_\et$ a weak integral étale Fourier transform if the following diagram commutes
\[
\begin{tikzcd}
    \CH_\et(A) \arrow{d}\arrow{r}{\mathcal{F}_\et} & \CH_\et(\widehat{A}) \arrow{d} \\
    \CH_\et(A)_\Q \arrow{r}{\mathcal{F}_A} & \CH_\et(\widehat{A})_\Q.
\end{tikzcd}
\]
We call a weak integral Fourier transform $\mathcal{F}_\et$ algebraic if it is induced by a cycle $\Gamma \in \CH_\et(A \times \widehat{A})$ that satisfies $\Gamma_\Q =\text{ch}(\mathcal{P}_A)$. 

% A group homomorphism  $\mathcal{F}_\et: \CH_\et(A)\to \CH_\et(\widehat{A})$ is an integral étale Fourier transform up to homology if the following diagram commutes:
% \[
% \begin{tikzcd}
%     \CH_\et(A) \arrow{d}\arrow{r}{\mathcal{F}_\et} & \CH_\et(\widehat{A}) \arrow{d} \\
%   \bigoplus_{i=0}^{2g} H^{2i}_\et(\widehat{A}_{k^s},\Z_\ell(i)) \arrow{r}{\mathfrak{F}_A} & \bigoplus_{i=0}^{2g} H^{2i}_\et(\widehat{A}_{k^s},\Z_\ell(i)).
% \end{tikzcd}
% \]
% Finally an integral étale Fourier transform up to homology $\mathcal{F}_\et$ is called algebraic if it is induced by a cycle $\Gamma \in \CH_\et(A\times \widehat{A})$ such that $\text{cl}(\Gamma) = \text{ch}(\mathcal{P}_A) \in \bigoplus_{i=0}^{4g} H^{2i}_\et((A \times \widehat{A})_{k^s},\Z_\ell(i))$. Similarly, a $\Z_\ell$-module homomorphism $\mathcal{F}_{\et,\ell}:\CH_\et(A)_{\Z_\ell} \to \CH_\et(\widehat{A})_{\Z_\ell}$ is called an $\ell$-adic integral Fourier transform up to homology if $\mathcal{F}_{\et,\ell}$ is compatible with $\mathfrak{F}_A$ and the $\ell$-adic cycle class map. If such homomorphism exists and is induced by a cycle $\Gamma_\ell \in \CH_\et(A\times \widehat{A})_{\Z_\ell}$ and $\text{cl}(\Gamma_\ell)=\text{ch}(\mathcal{P}_A)$ is called an algebraic $\ell$-adic integral étale Fourier transform.
\end{defi}

Following this notion of étale Fourier transform for étale motivic cohomology, one gets the following equivalence for the existence of a Fourier transform:

\begin{theorem*}[{Theorem \ref{teoEquiv}}]
 Let $A$ be an abelian variety over $k$ of dimension $g$. The following statements are equivalent:
 \begin{enumerate}
     \item The 1-cycle $\displaystyle \frac{c_1(\mathcal{P}_A)^{2g-1}}{(2g-1)!} \in \CH_\et(A \times \widehat{A})_\Q$ lifts to $\CH^{2g-1}_\et(A \times \widehat{A})$.
     \item The abelian variety $A$ admits an étale motivic weak integral Fourier transform.
     \item The abelian variety $A\times \widehat{A}$ admits an étale motivic weak integral Fourier transform.
 \end{enumerate}
 
 Suppose $A$ carries a symmetric ample line bundle, which induces a principal polarization $\lambda: A \xrightarrow{\sim} \widehat{A}$, then the previous statements are equivalent to the following
 \begin{enumerate}
     \item[(4)] The 2-cycle $\displaystyle \frac{c_1(\mathcal{P}_A)^{2g-2}}{(2g-2)!} \in \CH_\et(A \times \widehat{A})_\Q$ lifts to $\CH^{2g-2}_\et(A \times \widehat{A})$.
     \item[(5)] Denoting as $\Theta \in \CH_\et^1(A)_\Q$ to the symmetric ample class attached to $\lambda$, then the one cycle $\displaystyle \Gamma_\Theta= \frac{\Theta^{g-1}}{(g-1)!} \in \CH_\et(A)_\Q$ lifts to a one cycle in $\CH^{g-1}_\et(A)$.
     \item[(6)] The abelian variety $A$ admits a weak integral étale Fourier transform.
     \item[(7)] The Fourier transform $\mathcal{F}_A$ satisfies $\mathcal{F}_A(\CH_\et(A)_{tf})\subset \CH_\et(\widehat{A})_{tf}$.
     \item[(8)] There exists a PD-structure on the ideal $\CH_\et^{>0}(A)_{tf} \subset \CH_\et(A)_{\text{tf}}$.
 \end{enumerate}
\end{theorem*}

With the previous result one gets a characterization for the existence of the Fourier transform over any field. If we assume that $k=\bar{k}$, the result is unconditional, getting the existence of a morphism of étale cycles with integral coefficients:

\begin{theorem*}[{Corollary \ref{coro}}]
Let $k=\bar{k}$ and let $A/k$ be an abelian variety, then there exists an integral algebraic étale Fourier transform.
\end{theorem*}

\section*{Acknowledgments}

The author thanks to Fr\'ed\'eric D\'eglise and Johannes Nagel for their suggestions, useful discussions and the time for reading this article.  This work is supported by the EIPHI Graduate School (contract ANR-17-EURE-0002) and the FEDER/EUR-EiPhi Project EITAG. We thankthe French ANR project “HQ-DIAG” (ANR-21-CE40-0015).

\section{Etale Fourier transform}

\subsection{PD-structure}

For the sake of completeness, we should say a few words about the definition of étale motivic cohomology for non-smooth schemes over a field $k$. Given a Noetherian scheme $S$, we consider two classes of morphisms over $S$, the classes of finite type $\mathcal{F}^{\text{ft}}$ and étale morphism $\text{Ét}$. There exists a $\mathcal{F}^{\text{ft}}$ premotivic category $\underline{\text{DM}}_h(S,\Z)$ which is generated by h-motives $M^S_h(X)$ with $X/S$ a scheme of finite type over $S$, and with this we obtain an adjunction
\begin{align*}
    \nu_\#: \text{DM}_h(S,\Z)\leftrightarrows \underline{\text{DM}}_h(S,\Z):\nu^*
\end{align*}
This forms an enriched category of h-motives in the sense of \cite[Definition 1.4.13]{CD19} and \cite[Example 5.3.33]{CD19}. Thus, for a $k$-scheme of finite type $X$ we define the following groups
\begin{align*}
    \CH_\et^n(X):=\ho_{\underline{\text{DM}}_h(k,\Z)}\left(M_h(X),\mathbf{1}_k(n)[2n]\right), \quad   \CH^\et_n(X):=\ho_{\underline{\text{DM}}_h(k,\Z)}\left(M_h(X)(n)[2n],\mathbf{1}_k\right).
\end{align*}
Since $\nu_\#$ and $\nu^*$ commute with $f^*$ whenever we consider any morphism $f:X\to Y$, then we have that the push-forward is well defined for étale Chow groups defined by the enriched category of h-motives. If $X$ is smooth and projective over $k$ of dimension $d$, then the previously defined version of étale motivic cohomology corresponds with 
\begin{align*}
    \CH_\et^n(X)&\simeq \ho_{\text{DM}_h(k,\Z)}\left(M_h(X),\mathbf{1}_k(n)[2n]\right)\\
    &\simeq \ho_{\text{DM}_\et(k,\Z)}\left(M_\et(X),\mathbf{1}_k(n)[2n]\right)\\
    &\simeq \CH_{d-n}^\et(X).
    \end{align*}

We also breafly recall the construction of Lichtenbaum cohomology groups, groups defined by the hypercohomology of the \'etale sheafification of Bloch's complex sheaf. These groups are characterised by Rosenschon and Srinivas in \cite{RS} using \'etale hypercoverings. In this context we consider $\text{Sm}_k$ as the category of smooth separated $k$-schemes over a field $k$. We denote by $z^n(X,\bullet)$ the cycle complex of abelian groups defined by Bloch 
\begin{align*}
   z^n(X,\bullet): \cdots \to z^n(X,i) \to \cdots \to z^n(X,1)\to z^n(X,0) \to 0 
\end{align*}
where the differentials are given by the alternating sum of the pullbacks of the face maps and their homology groups define the higher Chow groups $\text{CH}^n(X,m)=H_m(z^n(X,\bullet))$.

Since  $z^n(X,i)$ and the complex $z^n(X,\bullet)$ are covariant functorial for proper maps and contravariant functorial for flat morphisms between smooth $k$-schemes, see \cite[Proposition 1.3]{Blo10}, therefore for a topology $t \in \left\{\text{fppf}, \ \et, \ \text{Nis},  \ \text{Zar} \right\}$ we have a complex of $t$-presheaves
$z^n(-,\bullet):U \mapsto z^n(U,\bullet)$. In particular the presheaf $z^n(-,i):U \mapsto z^n(U,i)$ is a sheaf for $t \in \left\{\text{fppf}, \ \et, \ \text{Nis},  \ \text{Zar} \right\}$, see \cite[Lemma 3.1]{Ge04}, and then $z^n(-,\bullet)$ is a complex of sheaves for the small \'etale, Nisnevich and Zariski sites of $X$. We set the complex of $t$-sheaves 
\begin{align*}
Z_X(n)_t = \left(z^n(-,\bullet)_t\right)[-2n]
\end{align*}
where $R$ is an abelian group and for our purposes we just consider $t= \text{Zar}$ or $\et$ and then we compute the hypercohomology groups $\mathbb{H}^m_t(X,R_X(n)_t)$. We denote the Lichtenbaum cohomology groups as
\begin{align*}
H_M^m(X,\Z(n))= \mathbb{H}_\text{Zar}^m(X,\Z(n)), \quad H_L^m(X,\Z(n))=\mathbb{H}_\et^m(X,\Z(n))
\end{align*}
and in particular we set $\text{CH}_L^n(X):=H^{2n}_L(X,\Z(n))$. For smooth projective $k$-variety $X$ with $p=\text{char}(k)$ one has $\mathbb{H}_\et^m(X,\Z(n))[1/p]\simeq H_{M,\et}^m(X,\Z(n))[1/p] $ where $A[1/p]:=A \otimes_Z Z[1/p].$ For the proof we refer to \cite[Theorem 7.1.11]{CD16}.

Let $X$ be a quasi-projective scheme over a field $k$. For $d\geq 1$ we define the $d$th symmetric power $\text{Sym}^d(X)$ of $X$ (over $k$) as the quotient of $X^d$ by the natural actions of the symmetric group $\mathfrak{S}_d$ (this quotient always exists for a finite group, see \cite[II, \S, nº 6 ]{DG} and  \cite[Exp. V, Cor. 1.5]{GR}). This quotient is functorial in the sense that for a morphism $f: X\to Y$ between quasi-projective $k$-schemes we have that $\text{Sym}^d(f):\text{Sym}^d(X) \to \text{Sym}^d(Y)$.

\begin{lemma}[{\cite[Lemma 1.1]{MP}}]
    Let $X$ be a quasi-projective scheme over $k$.
    \begin{enumerate}
        \item The quotient morphism $q_{d,X}:X^d \to \text{Sym}^d(X)$ is again quasi-projective.
        \item Assume that $X$ is equidimensional of dimension $n>0$. Then $X^d$ and $\text{Sym}^d(X)$ are equidimensional of dimension $dn$, and there exists a dense open subset in $\text{Sym}^d(X)$ over which $q_{d,X}$ is étale of degree $d!$.
        \item Assume that $X$ is equi-dimensional and that there exist non-negatives integers $d_1,\ldots,d_r$ such that $d_1+\ldots + d_r=d$. Then the natural map
        \begin{align*}
            \alpha_{d_1,\ldots,d_r}:\text{Sym}^{d_1}(X)\times \ldots \times \text{Sym}^{d_r}(X)\to \text{Sym}^{d}(X)
        \end{align*}
        is finite, and there is a dense open subset in $\text{Sym}^{d}(X)$ over which it is étale of degree $\displaystyle \frac{d!}{d_1!\cdot \ldots \cdot d_r!}$. For $d$, $e\geq 1$, the natural map $\text{Sym}^d(\text{Sym}^e(X)) \to \text{Sym}^{de}(X)$ is finite, and there is a dense open subset in $\text{Sym}^{de}(X)$ over which it is étale of degree $\displaystyle \frac{(de)!}{d!(e!)^d}$.
    \end{enumerate}
\end{lemma}

Let us consider a quasi-projective scheme $X$ over a field $k$. Consider the $d$-diagonal embedding $X\xrightarrow{\delta^d}X^d$ and the composite map $p_d: X\xrightarrow{\delta^d}X^d \xrightarrow{q_{d,X}}\text{Sym}^{d}(X)$. Since $\delta^d$ and $q_{d,X}$ are proper morphism, then we have a push-forward map
\begin{align*}
    (p_d)_*:\CH_m^\et(X) \to \CH_{dm}^\et(\text{Sym}^{d}(X))
\end{align*}
In the same way we define the Pontryagin product as
\begin{align*}
    \CH_*^\et(\text{Sym}^{d_1}(X))\times \CH_*^\et(\text{Sym}^{d_2}(X)) \to \CH_*^\et(\text{Sym}^{d_1+d_2}(X))
\end{align*}
using the formula $x*y:=(\alpha_{d_1,d_2})_*(x\times y)$. For a cycle $\xi=\sum_{j=1}^r n_j Z_j$ with $Z_j \in \CH_*^\et(X)$, we define 
\begin{align*}
\gamma_d(\xi):= \sum_{d_1+\ldots+d_r=d} n_1^{d_1}\cdots n_r^{d_r}\cdot \gamma_{d_1}(Z_1)* \ldots * \gamma_{d_r}(Z_r)
\end{align*}
For $d=0$ let us set for an element $a \in \CH_*^\et(X)$ the $\gamma_0(a)=[\spc(k)] \in \CH_0^\et(X)$, which is the unit element in $\CH_0^\et(X)$.

\begin{lemma}\label{lemmagamma}
    Let $X$ be a quasi-projective scheme over $k$, then:
    \begin{enumerate}
        \item  If $Z \in \CH_{>0}^\et(X)$ and $d_1,\ldots,d_t$ are non-negative integers with $d_1+\ldots+d_t=d$ then
        \begin{align*}
            \gamma_{d_1}(Z)*\cdots*\gamma_{d_t}(Z) = \frac{d!}{d_1!d_2!\cdots d_t!} \cdot \gamma_d(Z).
        \end{align*}
        \item Let $\xi_1,\ldots,\xi_t$ be cycles in $\CH_{>0}^\et(X)$, and let $\xi=\sum_{i=1}^t \xi_i$. Then
        \begin{align*}
            \gamma_d(\xi)= \sum_{d_1+\ldots+d_t=d}  \gamma_{d_1}(\xi_1)*\cdots*\gamma_{d_t}(\xi_t).
        \end{align*}
        \item If $i: V \hookrightarrow X$ is a closed immersion and let $\xi$ be a cycle in $\CH_m^\et(V)$, then
        \begin{align*}
            \gamma_d(i_*\xi)=\text{Sym}^d(i)_*(\gamma_d(\xi)) \in \CH_{dm}^\et(\text{Sym}^d(X))
        \end{align*}
        \item Let $V\subset X$ be a closed subscheme, equidimensional of positive dimension. Then 
        \begin{align*}
            \gamma_d([V]_\et)=[\text{Sym}^d(V)]_\et
        \end{align*}
        where we view $\text{Sym}^d(V)$ as a closed subscheme of $\text{Sym}^d(X)$ and $[\text{Sym}^d(V)]_\et$ represents the image of the algebraic cycle $[\text{Sym}^d(V)]\in \CH_{dm}(\text{Sym}^d(X))$ in $\CH_{dm}^\et(\text{Sym}^d(X))$  
    \end{enumerate}
\end{lemma}

\begin{proof}
(1) Is a direct consequence of \cite[Lemma 1.1.(iii)]{MP}, \cite[Proposition 11.2.5]{CD19} and the functor $\text{DM}(X,\Z) \to \text{DM}_\et(X,\Z)$. The same argument works for (2) and the compatibility of the comparison map.

(3) This again is obtained by the definition of Pontryagin product using push-forward. The push-forward via the map $\text{Sym}^d(i): \text{Sym}^d(V)\to  \text{Sym}^d(X)$ respects Pontryagin products.

(4) Is a direct consequence of \cite[Lemma 1.3.4]{MP} and the compatibility of the comparison map with proper push-forwards.    
\end{proof}

\begin{lemma}
    Let $f:X\to Y$ be a proper morphism of quasi-projective $k$-schemes. Then for all $x\in \CH_{>0}^\et(X)$ and all $d\geq 0$ one has 
    \begin{align*}
        \text{Sym}^d(f)_* (\gamma_d(x)) = \gamma_d(f_* x)
    \end{align*}
\end{lemma}

\begin{proof}
This follows from \cite[Proposition 1.5]{MP} and \cite[Proposition 11.2.5]{CD19}.   
\end{proof}

Let $k$ be a field and let $(M_n)_{n\in \N}$ be a commutative graded monoid in the category of quasi-projective schemes. Recalling such definition: $M_n$ is a quasi-projective $k$-scheme for all $n \geq 0$, and that we have product maps $\mu_{m,n}: M_m\times M_n \to M_{m+n}$ which satisfy commutativity and associativity. Assuming that there exists a $k$-rational point $e \in M_0(k)$ which is a unit for these products and that the maps $\mu_{n,m}$ are proper morphism, then we can define the Pontryagin product on the ring
\begin{align*}
    \CH_*^\et(M_\bullet):= \bigoplus_{n \in \N} \CH_*^\et(M_n)
\end{align*}
by the formula $x * y :=(\mu_{m,n})_*(x \times y)$ for $x \in \CH_*^\et(M_m)$ and $y \in \CH_*^\et(M_n)$. Something that we ought to notice is that the iteration of the multiplication map  $\mu_{n,\ldots,n}:M^d_n \to M_{dn}$ factors through the proper map $p_d: \text{Sym}^d(M_n) \to M_{dn}$. We set $p_0$ to be the map $e \to M_0$. 

\begin{theorem}
    For a commutative graded monoid $(M_n)_{n \in \N}$ with identity and with proper product morphisms, the maps
    \begin{align*}
        \gamma_d^M: \CH_{>0}^\et(M_n) \to \CH_*^\et(M_{dn})
    \end{align*}
    given by $x\mapsto (p_d)_* \gamma_d(x)$ extends uniquely to a PD-structure $\{\gamma_d^M\}_{d\geq 0}$ on the ideal $\CH_{>0}^\et(M_\bullet)\subset \CH_*^\et(M_\bullet)$. The PD-structure is functorial with respect to $(f_n:M_n \to N_n)$ which are proper for all $n \in \N$.
\end{theorem}
\begin{proof}
    Let $x = \sum_{n \in \N} x_n$, $x_n \in \CH_{>0}^\et(M_\bullet)$ and $x_n$ is non-zero for finitely many $n$. Thus we define 
    \begin{align*}
        \gamma_d^M(x):=\sum_{d_1+d_2+\ldots=d} \gamma_{d_1}^M(x_1)*\gamma_{d_2}^M(x_2)*\ldots
    \end{align*}
Clearly by definition we get $\gamma_d^M(\lambda x)= \lambda^d \gamma_d^M(x)$, and by Lemma \ref{lemmagamma} we have
\begin{align*}
    \gamma_d^M(x+y) = \sum_{d_1+d_2=d} \gamma_{d_1}^M(x)*\gamma_{d_2}^M(y)
\end{align*}
for all $x$, $y \in \CH_{>0}^\et(M_\bullet)$ and for all $d\geq 0$. As this formula holds, for $x \in \CH_{>0}^\et(X)$ and  $d$, $e \geq 0$ we obtain the following relation
\begin{align*}
    \gamma_d^M(x)*\gamma_e^M(x) = \binom{d+e}{d}\cdot \gamma_{d+e}^M(x).
\end{align*}

For the other property 
\begin{align*}
    \gamma_d^M\left(\gamma_e^M(x)\right) = \frac{(de)!}{d!(e!)^d}\gamma_{de}^M(x)
\end{align*}
we apply \cite[Proposition 11.2.5]{CD19}, the functorial properties of the topology change $\rho: X_\et \to X_{\text{Zar}}$ and \cite[Lemma 1.1.(iii)]{MP}. 
\end{proof}

If $X$ is a smooth quasi-projective $k$-scheme then the previous construction gives a PD-structure on the graded ideal $\bigoplus_{i\geq 0}\CH_{>0}^\et(\text{Sym}^i(X))$. If $M_n = \emptyset$ for all $n>0$ (like for example an abelian variety and the multiplication on it) we obtain the following version (which is the ungraded version):

\begin{corollary}\label{coro1}
Let $M$ be a commutative monoid with identity in the category of quasi-projective $k$-schemes, such that the product morphism $\mu : M\times M \to M$ is proper. Let $p_d:\text{Sym}^d(M)\to M$ be the morphism induced by the iterated multiplication map $M^d \to M$. Then the maps $\gamma_d^M: \CH_{>0}^\et(M) \to \CH_*^\et(M)$ defined by $x \mapsto (p_d)_* \gamma_d(x)$ define a PD-structure on the ideal $\CH_{>0}^\et(M)\subset \CH_*^\et(M)$.    
\end{corollary}

\begin{corollary}\label{coro2}
Let $k = \bar{k}$ be a field. Let $A$ be an abelian variety over $k$, then there is a canonical PD-structure, with respect to the Pontryagin product, on the augmentation ideal in $\CH_*^\et(A)$, generated by $\CH_{>0}^\et(A)$ together with the $0$-cycles of degree 0.
\end{corollary}
\begin{proof}
Let $I \subset \CH_0^\et(A)$ be the ideal of $0$-cycles of degree 0 on $A$. Since $\CH_0(A)\simeq \CH_0^L(A)$ over algebraically closed fields, then we conclude as in \cite[Corollary 1.8]{MP}.  
\end{proof}

\subsection{Etale Fourier transform}
Let $A$ be an abelian variety over a field $k$. The Fourier transform on the level of Chow groups is the groups homomorphism 
\begin{align*}
    \mathcal{F}_A: \CH(A)_\Q \to \CH(\widehat{A})_\Q
\end{align*}
induced by the correspondence $\text{ch}(\mathcal{P}_A)\in \CH(A\times \widehat{A})_\Q$, where $\text{ch}(\mathcal{P}_A)$ is the Chern character of $\mathcal{P}_A$. One has the Fourier transform on the level of étale cohomology:
\begin{align*}
    \mathfrak{F}_A: H^\bullet_\et(A_{k^s},\Q_\ell(\bullet)) \to H^\bullet_\et(\widehat{A}_{k^s},\Q_\ell(\bullet))
\end{align*}
which preserves integral cohomology classes and induces, for each $i$ with $0 \leq i \leq 2g$, an isomorphism
\begin{align*}
    \mathfrak{F}_A: H^i_\et(A_{k^s},\Z_\ell(n)) \to H^{2g-i}_\et(\widehat{A}_{k^s},\Z_\ell(n+g-i)),
\end{align*}
and if $k=\C$, then $\text{ch}(\mathcal{P}_A)$ induces, for each $0 \leq i \leq 2g$, an isomorphism of Hodge structures
\begin{align*}
    \mathfrak{F}_A: H^i(A,\Z) \to H^{2g-i}(\widehat{A},\Z(g-i)).
\end{align*}

\begin{defi}
    Let $A$ be an abelian variety over $k$ and let $\mathcal{F}_\et: \CH_\et(A)\to \CH_\et(\widehat{A})$ be a group homomorphism. We call $\mathcal{F}_\et$ a weak integral étale Fourier transform if the following diagram commutes
\[
\begin{tikzcd}
    \CH_\et(A) \arrow{d}\arrow{r}{\mathcal{F}_\et} & \CH_\et(\widehat{A}) \arrow{d} \\
    \CH_\et(A)_\Q \arrow{r}{\mathcal{F}_A} & \CH_\et(\widehat{A})_\Q.
\end{tikzcd}
\]
We call a weak integral Fourier transform $\mathcal{F}_\et$ algebraic if it is induced by a cycle $\Gamma \in \CH_\et(A \times \widehat{A})$ that satisfies $\Gamma_\Q =\text{ch}(\mathcal{P}_A)$. A group homomorphism  $\mathcal{F}_\et: \CH_\et(A)\to \CH_\et(\widehat{A})$ is an integral étale Fourier transform up to homology if the following diagram commutes:
\[
\begin{tikzcd}
    \CH_\et(A) \arrow{d}\arrow{r}{\mathcal{F}_\et} & \CH_\et(\widehat{A}) \arrow{d} \\
   \bigoplus_{i=0}^{2g} H^{2i}_\et(\widehat{A}_{k^s},\Z_\ell(i)) \arrow{r}{\mathfrak{F}_A} & \bigoplus_{i=0}^{2g} H^{2i}_\et(\widehat{A}_{k^s},\Z_\ell(i)).
\end{tikzcd}
\]
Finally an integral étale Fourier transform up to homology $\mathcal{F}_\et$ is called algebraic if it is induced by a cycle $\Gamma \in \CH_\et(A\times \widehat{A})$ such that $\text{cl}(\Gamma) = \text{ch}(\mathcal{P}_A) \in \bigoplus_{i=0}^{4g} H^{2i}_\et((A \times \widehat{A})_{k^s},\Z_\ell(i))$. Similarly, a $\Z_\ell$-module homomorphism $\mathcal{F}_{\et,\ell}:\CH_\et(A)_{\Z_\ell} \to \CH_\et(\widehat{A})_{\Z_\ell}$ is called an $\ell$-adic integral Fourier transform up to homology if $\mathcal{F}_{\et,\ell}$ is compatible with $\mathfrak{F}_A$ and the $\ell$-adic cycle class map. If such homomorphism exists and is induced by a cycle $\Gamma_\ell \in \CH_\et(A\times \widehat{A})_{\Z_\ell}$ and $\text{cl}(\Gamma_\ell)=\text{ch}(\mathcal{P}_A)$ is called an algebraic $\ell$-adic integral étale Fourier transform.
\end{defi}

If $\mathcal{F}_\et: \CH_\et(A)\to \CH_\et(\widehat{A})$ is a weak integral étale Fourier transform, then $\mathcal{F}_\et$ is an integral étale Fourier transform up to homology. If $k=\C$, then $\mathcal{F}_\et :\CH_\et(A)\to \CH_\et(\widehat{A})$ is an integral étale Fourier transform up to homology if and only if $\mathcal{F}_\et$ is compatible with the Fourier transform $\mathfrak{F}_A:H_B^\bullet (A,\Z) \to H_B^\bullet (\widehat{A},\Z)$ on Betti cohomology.

\begin{lemma} 
    Let $A$ be a complex abelian variety and let $\mathcal{F}_\et: \CH_\et(A)\to \CH_\et(\widehat{A})$ be an integral étale Fourier transform up to homology.
    \begin{enumerate}
        \item For each $i \in \N$ the integral étale Hodge conjecture for degree $2i$ classes on $A$ implies the integral étale Hodge conjecture for degree $2(g-i)$ classes on $\widehat{A}$.
        \item If $\mathcal{F}_\et$ is algebraic, then $\mathfrak{F}_A$ induces a group isomorphism $Z_\et^{2i}(A) \to Z_\et^{2(g-i)}(\widehat{A})$, where $Z_\et^{2i}(A)$ is the image of the Lichtenbaum cycle class map.
    \end{enumerate}
\end{lemma}

\begin{proof}
    Consider the following diagram
    \[
    \begin{tikzcd}
        \CH_\et^i(A) \arrow{r}\arrow{d}{c^i_\et}& \CH_\et(A) \arrow{r}{\mathcal{F}_\et}\arrow{d}& \CH_\et(\widehat{A}) \arrow{d} \arrow{r}& \CH^{g-i}_\et(\widehat{A}) \arrow{d}{c^{g-i}_\et} \\ 
        H^{2i}_B(A,\Z) \arrow{r}& H^\bullet_B(A,\Z) \arrow{r} & H^\bullet_B(\widehat{A},\Z) \arrow{r} & H^{2(g-i)}_B(\widehat{A},\Z)
    \end{tikzcd}
    \]
The composition of the bottom line $H^{2i}_B(A,\Z) \to H^{2(g-i)_B}(\widehat{A},\Z)$ is an isomorphism of Hodge structures, then  we obtain a commutative diagram
\[
\begin{tikzcd}
    \CH^i_\et(A) \arrow{r} \arrow{d}{c^i_\et} & \CH^{g-i}_\et(\widehat{A}) \arrow{d}{c^{g-i}_\et} \\
\text{Hdg}^{2i}(A,\Z) \arrow{r}{\simeq} & \text{Hdg}^{2(g-i)}(\widehat{A},\Z)
\end{tikzcd}
\]
Thus the surjectivity of $c^i_\et$ implies the surjectivity of $c^{g-i}_\et$. Arguing in the same way for $\widehat{A}$ and $\doublehat{A}$ we obtain the desired equivalence.
\end{proof}

For an abelian variety $A$ over $k$ we define the following cycles:
\begin{align*}
    \ell = c_1(\mathcal{P}_A) \in \CH^1_\et(A \times \widehat{A})_\Q, \quad  & \widehat{\ell} = c_1(\mathcal{P}_{\widehat{A}}) \in \CH^1_\et(\widehat{A} \times A)_\Q  \\
    \mathcal{R}_A = \frac{c_1(\mathcal{P}_A)^{2g-1}}{(2g-1)!} \in \CH^{2g-1}_\et(A \times \widehat{A})_\Q, \quad  &
   \mathcal{R}_{\widehat{A}} = \frac{c_1(\mathcal{P}_{\widehat{A}})^{2g-1}}{(2g-1)!} \in \CH^{2g-1}_\et(\widehat{A} \times A)_\Q.
\end{align*}
For $a \in \CH_\et(A)_\Q$ we define $E(a)\in \CH_\et(A)_\Q$ as the exponential element using $*$-operation:
\begin{align*}
    E(a):= \sum_{n \geq 0}\frac{a^{*n}}{n!} \in \CH_\et(A)_\Q.
\end{align*}

The following theorem is the same one as \cite[Theorem 3.8]{BG} but changing Chow groups to its étale analogue, 

\begin{theorem}\label{teoEquiv}
 Let $A$ be an abelian variety over $k$ of dimension $g$. The following statements are equivalent:
 \begin{enumerate}
     \item The one cycle $\displaystyle \frac{c_1(\mathcal{P}_A)^{2g-1}}{(2g-1)!} \in \CH_\et(A \times \widehat{A})_\Q$ lifts to $\CH^{2g-1}_\et(A \times \widehat{A})$.
     \item The abelian variety $A$ admits an étale motivic weak integral Fourier transform.
     \item The abelian variety $A\times \widehat{A}$ admits an étale motivic weak integral Fourier transform.
 \end{enumerate}
 
 If  we assume that $A$ carries a symmetric ample line bundle which induces a principal polarization $\lambda: A \xrightarrow{\sim} \widehat{A}$, therefore the previous statements are equivalent to the following
 \begin{enumerate}
     \item[(4)] The two cycle $\displaystyle \frac{c_1(\mathcal{P}_A)^{2g-2}}{(2g-2)!} \in \CH_\et(A \times \widehat{A})_\Q$ lifts to $\CH^{2g-2}_\et(A \times \widehat{A})$.
     \item[(5)] Denoting as $\Theta \in \CH_\et^1(A)_\Q$ to the symmetric ample class attached to $\lambda$, then the one cycle $\displaystyle \Gamma_\Theta= \frac{\Theta^{g-1}}{(g-1)!} \in \CH_\et(A)_\Q$ lifts to a one cycle in $\CH^{g-1}_\et(A)$.
     \item[(6)] The abelian variety $A$ admits a weak integral étale Fourier transform.
     \item[(7)] The Fourier transform $\mathcal{F}_A$ satisfies $\mathcal{F}_A(\CH_\et(A)_{tf})\subset \CH_\et(\widehat{A})_{tf}$.
     \item[(8)] There exists a PD-structure on the ideal $\CH_\et^{>0}(A)_{tf} \subset \CH_\et(A)_{\text{tf}}$.
 \end{enumerate}
\end{theorem}

\begin{proof}
    Assuming (1), then there exists a cycle $Z \in \CH_\et^{2g-1}(A \times \widehat{A})$ such that $Z_\Q \in \CH_\et^{2g-1}(A \times \widehat{A})_\Q$ equals $\displaystyle \frac{c_1(\mathcal{P}_A)^{2g-1}}{(2g-1)!}$. Consider the cycle $(-1)^g\cdot E((-1)^g\cdot Z) \in \CH_\et(A \times \widehat{A})$, by \cite[Lemma 3.4]{BG} we have that
    \begin{align*}
        (-1)^g\cdot E((-1)^g\cdot Z)_\Q =  (-1)^g\cdot E\left((-1)^g\cdot \frac{c_1(\mathcal{P}_A)^{2g-1}}{(2g-1)!}\right) = \text{ch}(\mathcal{P}_A) \in \CH_\et(A \times \widehat{A})_\Q  
    \end{align*}
    then follows (2). By the same principle, the line bundle $\mathcal{P}_{A \times \widehat{A}}$ on the abelian variety $X=A \times  \widehat{A} \times  \widehat{A} \times A$, we have that $\mathcal{P}_{A \times \widehat{A}}\simeq \pi^*_{13}\mathcal{P}_{A}\otimes \pi^*_{24}\mathcal{P}_{\widehat{A}}$,  then 
    \begin{align*}
        \mathcal{R}_{A\times \widehat{A}}&=\frac{(\pi^*_{13}c_1(\mathcal{P}_{A})+\pi^*_{24}c_1(\mathcal{P}_{\widehat{A}}))^{4g-1}}{(4g-1)!} \\
        &=\pi_{13}^*\left(\frac{c_1(\mathcal{P}_{A})^{2g-1}}{(2g-1)!} \right)\cdot \pi_{24}^*([0]_{A\times \widehat{A}})+\pi_{13}^*([0]_{\widehat{A}\times A})\cdot \pi_{24}^*\left(\frac{c_1(\mathcal{P}_{\widehat{A}})^{2g-1}}{(2g-1)!} \right)
    \end{align*}
therefore we conclude that $\mathcal{R}_{A\times \widehat{A}}$ lifts to $\CH^{4g-1}_\et(X)$, this implies that $A\times \widehat{A}$ admits a motivic weak integral Fourier transform. (3)$\Longrightarrow$(1) follows from the fact that $(-1)^g\mathcal{F}_{\widehat{A}\times A}(-\widehat{\ell})=\mathcal{R}_{A}$.

From now on, we assume that $A$ is a principally polarized variety $\lambda: A \to \widehat{A}$, with $\mathcal{L}$ be the symmetric ample line bundle. Assuming that (4) holds and denoting $s_A \in \CH_2^{\et}(A\times A)= \CH^{2g-2}_\et(A\times A)$ such that $(s_A)_{\Q}=\frac{c_1(\mathcal{P}_A)^{2g-2}}{(2g-2)!}$. Consider the symmetric line bundles $\CH^1_{\text{Sym}}(A) \subset \CH^1(A)$ and the homomorphism $\mathcal{F}:\CH^1_{\text{Sym}}(A)\to  \CH^\et_1(A)$ defined as the composition
\begin{align*}
    \CH^1_{\text{Sym}}(A) \hookrightarrow \CH^1(A) \xrightarrow{\text{pr}_1^*} \CH^1_\et(A\times A)\xrightarrow{\cdot s_A} \CH^{2g-2}_\et(A\times A)\xrightarrow{\text{pr}_{2*}} \CH_1^\et(A)
\end{align*}
As the line bundle $\mathcal{L}$ is symmetric, we have the following
\begin{align*}
    \Theta &= \frac{1}{2}\cdot (\text{id},\lambda)^* c_1(\mathcal{P}_A)\\
    &= \frac{1}{2}\cdot c_1((\text{id},\lambda)^*\mathcal{P}_A) \\
    &=\frac{1}{2}\cdot c_1(\mathcal{L}\otimes \mathcal{L}) =c_1(\mathcal{L}) \in \CH^1(A)_\Q
\end{align*}
The Chern class of $\mathcal{L}$ is sent to $\Theta$, therefore $\mathcal{F}(c_1(\mathcal{L}))_\Q = \Gamma_\Theta$, therefore (5) holds. If we assume that (5) holds, then by \cite[Lemma 3.5]{BG} we obtain (1).

If (2) holds then immediately holds (4), so we obtain that $(4)\Longrightarrow(5)\Longrightarrow(1)\Longrightarrow(2)\Longrightarrow(4)$. Under the assumptions about polarization, we see that $(2)\Longrightarrow(6)\Longrightarrow(7)$. If we assume that (7), since $\Theta = c_1(\mathcal{L})$ is lifted to $\CH^1(A)$, then $\mathcal{F}_A(\Theta)=(-1)^{g-1}\Gamma_\Theta$ is lifted to $\CH_1^\et(A)$, thus (5) holds. Again if (7) holds, then $\mathcal{F}_A$ defines an isomorphism 
\begin{align*}
    \mathcal{F}_A: \CH_\et(A)_{\text{tf}} \xrightarrow{\sim} \CH_\et(A)_{\text{tf}}. 
\end{align*}
The ideal $\CH^\et_{>0}(A)_{\text{tf}} \subset \CH_\et(A)_{\text{tf}} $ admits a PD-structure for the Pontryagin product. As $\mathcal{F}_A$ exchanges Pontryagin product by intersection product, we obtain (8).   
\end{proof}

Using the same arguments, we can obtain the following equivalences for the different notions of étale Fourier transform:
\begin{prop}\label{propEquiv}
Let $A$ be an abelian variety of dimension $g$ over a field $k$. The following assertions are equivalent:
\begin{enumerate}
    \item The class $\displaystyle \frac{c_1(\mathcal{P}_A)^{2g-1}}{(2g-1)!} \in H^{4g-2}_\et((A \times \widehat{A})_{k^s},\Z_\ell(2g-1))$ is the class of a cycle in $\CH^{2g-1}_\et(A \times \widehat{A})$.
    \item The abelian variety $A$ admits an étale integral Fourier transform up to homology which is algebraic.
    \item The abelian variety $A\times \widehat{A}$ admits an étale integral Fourier transform up to homology which is algebraic.
\end{enumerate}
If  we assume that $A$ carries a symmetric ample line bundle which induces a principal polarization $\lambda: A \xrightarrow{\sim} \widehat{A}$, therefore the previous statements are equivalent to the following:
\begin{enumerate}
    \item[(4)] The class $\displaystyle \frac{c_1(\mathcal{P}_A)^{2g-2}}{(2g-2)!} \in H^{4g-4}_\et((A \times \widehat{A})_{k^s},\Z_\ell(2g-2))$ is the class of a cycle in $\CH^{2g-2}_\et(A \times \widehat{A})$.
    \item[(5)] The class $\theta^{g-1}/(g-1)! \in H^{2g-2}_\et(A_{k^s},\Z_\ell(g-1))$ lifts to a cycle in $\CH^{g-1}_\et(A)$.
    \item[(6)] The abelian variety $A$ admits an integral étale Fourier transform up to homology.
\end{enumerate}

If $k=\C$ then the previous (1)-(6) is equivalent to the same statement replacing étale cohomology by Betti cohomology.
\end{prop}

\begin{proof}
The proof of the equivalence is analogous to the one in Theorem \ref{teoEquiv}. If the base field is $k=\C$, we have an isomorphism $H_\et^i(A,\Z_\ell) \simeq H_B^i(A,\Z_\ell)$ and the fact that $\beta \in H^{2i}_B(A,\Z)$ is in the image of the cycle class map if and only if $\beta_\ell \in H^{2i}_\et(A,\Z_\ell)$ is in the image of the cycle class map.
\end{proof}

\begin{remark}
Since the PD-sctructure on $(\CH^\et_{>0}(A),*)$ induces a PD-structure on $(\CH^\et_{>0}(A)_{\Z_\ell},*)$, Proposition \ref{propEquiv} remains true if we replace $\CH_\et(A)$ by $\CH_\et(A)_{\Z_\ell}$ and ``étale integral Fourier transform up to homology" by ``étale $\ell$-adic integral Fourier transform up to homology".
\end{remark}

\begin{corollary}\label{coroAlg}
Let $k$ be one of the following fields: $\C$, finitely generated over $\Q$, or $\mathbb{F}_{p^r}$ for $p$ a prime and $r \in \N$. Then every abelian variety $A/k$ admits an étale algebraic integral Fourier transform up to homology in the complex case and an étale $\ell$-adic integral Fourier transform up to homology for the other cases.
\end{corollary}
\begin{proof}
For the complex case this is a consequence of \cite[Theorem 1.1]{RS} and the étale analog of \cite[Corollary 4.1]{BG}. The other cases are a consequence of \cite[Theorem 1.3 and 1.4]{RS} respectively and \cite[Corollary 6.1]{BG} using the formula of $\mathcal{R}_{A\times \widehat{A}}$ in proof of Theorem \ref{teoEquiv}.
\end{proof}

\begin{corollary}{\label{FourierTran}}
Let $k$ be an algebraically closed field and let $\ell\neq \text{char}(k)$ be a prime integer. Given an abelian variety $A/k$, then for $\ell$ there exists an $\ell$-adic integral étale Fourier transform up to homology which is algebraic.
\end{corollary}

\begin{proof}
Consider a smooth projective variety $X$ over $k$, then we have a cycle class map $c_{\et,\ell}^i: \CH^i_\et(X)_{\Z_\ell} \to H^{2i}_\et(X,\Z_\ell)$. Consider a finitely generated sub $\Z_\ell$-module $G \subseteq H^{2i}_\et(X,\Z_\ell)$. Let $\CH^i_\et(X)_{\Z_\ell}\supseteq W:=c^{i,-1}_{\et,\ell}(G)$ and take the map $f$ as $c^i_{\et,\ell}$ restricted to $W$, thus we have $f: W \to G$. Denoting by  $I_{W,\et,\ell}^{2i}(X)=\text{im}(f)$ and $I_{\et,\ell}^{2i}(X):\text{im}(c^{i}_{\et,\ell})$, then we have that $\left(G/I_{W,\et,\ell}^{2i}(X)\right)\{\ell\}\hookrightarrow \left(H^{2i}_\et(X,\Z_\ell)/I_{\et,\ell}^{2i}(X)\right)\{\ell\}=0$. Thus we can conclude that $\left(G/I_{W,\et,\ell}^{2i}(X)\right)$ is a torsion free $\Z_\ell$-module, thus 
\begin{align*}
    \left(G/I_{W,\et,\ell}^{2i}(X)\right)\otimes \Q =0 \iff \left(G/I_{W,\et,\ell}^{2i}(X)\right)=0
\end{align*}
So this implies that $G$ is in the preimage of $\CH^i_\et(X)_{\Z_\ell}$ if and only if  $G\otimes \Q_\ell$ is in the preimage of $\CH^i_\et(X)_{\Q_\ell}$. In particular, consider $X=A \times \widehat{A}$ and the integral $\ell$-adic class $\displaystyle \frac{c_1(\mathcal{P}_A)^{2g-1}}{(2g-1)!} \in H^{4g-2}_\et(A \times \widehat{A},\Z_\ell(2g-1))$ and let $G$ be the  $\Z_\ell$-submodule of $H^{4g-2}_\et(A \times \widehat{A},\Z_\ell(2g-1))$ generated by $\displaystyle \frac{c_1(\mathcal{P}_A)^{2g-1}}{(2g-1)!}$ (which is algebraic with rational coefficients), thus by the previous remark we get that it lifts to $\CH_1^\et(A \times \widehat{A})_{\Z_\ell}$.
\end{proof}

\begin{remark}
Although we can lift the Chern class of the Poincaré bundle $\mathcal{P}_A$, it is not clear whether the Tate conjecture holds for abelian varieties due to the obstruction $T_\ell(\text{Br}^n(A \times \widehat{A}))$.
\end{remark}

\begin{theorem}\label{teoAlg}
Let $k$ be an algebraically closed field, then for any abelian variety $A$ over $k$ there exists an integral algebraic étale Fourier transform if and only if for all $\ell \neq \text{char}(k)$, $A/k$ admits an étale $\ell$-adic integral Fourier transform up to homology.
\end{theorem}

\begin{proof}
One way is clear. For the other one, we will split the proof into several parts: first, we prove that the Fourier transform preserves torsion classes. Note that since $H^i_\et(A,\Z_\ell(j))$ is a torsion-free $\Z_\ell$ module, then we have a short exact sequence
\begin{align*}
    0 \to H^i_\et(A,\Z_\ell(j)) \to H^i_\et(A,\Q_\ell(j)) \to H^i_\et(A,\Q_\ell/\Z_\ell(j)) \to 0,
\end{align*}
assuming that $i\neq 2j+1$ by \cite[Proposition 5.1]{RS} then we have an isomorphism $H^i_\et(A,\Q_\ell/\Z_\ell(j)) \simeq H^{i+1}_{M,\et}(A,\Z(j))\{\ell\}$, the same holds for $\widehat{A}$. With this we have obtain a commutative diagram
\[
\begin{tikzcd}
  0 \arrow{r } & H^i_\et(A,\Z_\ell(j))\arrow{r }\arrow{d}{\mathfrak{F}_A} & H^i_\et(A,\Q_\ell(j)) \arrow{d}{\mathfrak{F}_A}\arrow{r } & H^i_\et(A,\Q_\ell/\Z_\ell(j)) \arrow{d}{\mathfrak{F}^{q,\ell}_A}\arrow{r }& 0 \\
    0 \arrow{r } & H^{2g-i}_\et(\widehat{A},\Z_\ell(a))\arrow{r } & H^{2g-i}_\et(\widehat{A},\Q_\ell(a)) \arrow{r } & H^{2g-i}_\et(\widehat{A},\Q_\ell/\Z_\ell(a)) \arrow{r }& 0
\end{tikzcd}
\]
where $a=j+g-i$ and $\mathfrak{F}^{q,\ell}_A$ is the induced map by the quotient, so we have an morphism of torsion groups, thus we have that
$\mathfrak{F}^{q,\ell}_A:\CH^{i}_\et(A)\{\ell\} \xrightarrow{\simeq}\CH^{g-i+1}_\et(\widehat{A})\{\ell\}$. 
Assuming that for each prime number $\ell\neq \text{char}(k)$, then we have a commutative diagram
    \[
    \begin{tikzcd}
        \CH_\et^i(A)_{\Z_\ell} \arrow{r}\arrow{d}{c^i_{\et,\ell}}& \CH_\et(A)_{\Z_\ell} \arrow{r}{\mathcal{F}_\et}\arrow{d}& \CH_\et(\widehat{A})_{\Z_\ell} \arrow{d} \arrow{r}& \CH^{g-i}_\et(\widehat{A})_{\Z_\ell} \arrow{d}{c^{g-i}_{\et,\ell}} \\ 
        H^{2i}_\et(A,\Z_\ell) \arrow{r}& H^\bullet_\et(A,\Z_\ell) \arrow{r} & H^\bullet_\et(\widehat{A},\Z_\ell) \arrow{r} & H^{2(g-i)}_\et(\widehat{A},\Z_\ell)
    \end{tikzcd}
    \]
swapping the Fourier transform $\widehat{A}$ with the double dual of $A$, we can conclude that $c^i_{\et,\ell}(\CH_\et^i(A)_{\Z_\ell})\simeq c^{g-i}_{\et,\ell}(\CH_\et^{g-i}(\widehat{A})_{\Z_\ell})$. Since the kernel of the cycle class map $c^i_{\et,\ell}$ is $\ell$-divisible, then $\CH_\et^i(A)\otimes\Q_\ell/\Z_\ell\simeq I^{2i}_{\et,\ell}(A)\otimes\Q_\ell/\Z_\ell$.

Consider the short exact sequence $0 \to \Z_\ell \to \Q_\ell \to \Q_\ell/\Z_\ell \to 0 $, then by the previous remark, we obtain a quotient map $F^\ell_A:\CH_\et(A)_{\Q_\ell/\Z_\ell} \to \CH_\et(\widehat{A})_{\Q_\ell/\Z_\ell}$ which is an isomorphism. First we assume that $\text{char}(k)=0$, which gives us a commutative diagram
\[
\begin{tikzcd}
  0 \arrow{r } &\CH_\et(A)_{\text{tors}}\arrow{d}{\mathfrak{F}_A^q}\arrow{r} & \CH_\et(A)  \arrow{d}{\mathfrak{F}_A}\arrow{r} &\CH_\et(A)_{\Q} \arrow{d}{\mathcal{F}_A}\arrow{r}& \arrow{r}\CH_\et(A)_{\Q/\Z}\arrow{d}{F_A} & 0\\
    0 \arrow{r} & \CH_\et(\widehat{A})_{\text{tors}}\arrow{r} & \CH_\et(\widehat{A}) \arrow{r} & \CH_\et(\widehat{A})_{\Q} \arrow{r}&\CH_\et(\widehat{A})_{\Q/\Z}  \arrow{r}&0
\end{tikzcd}
\]
where $\displaystyle \mathfrak{F}_A^q= \bigoplus_{\ell \neq \text{char}(k)} \mathfrak{F}_A^{q,\ell}$ and $\displaystyle  F_A= \bigoplus_{\ell \neq \text{char}(k)} F_A^{\ell}$. In particular we found $\mathfrak{F}_A \otimes \Q= \mathcal{F}_A$, so it preserves integral étale cycles. If we work over positive characteristic $p$, then we take $\ell\neq p$ and use the fact that $\CH_\et(A)\simeq \CH_L(A)[1/p]$ where $\CH_L(A)$ are the Lichtenbaum cohomology groups of $A$.
\end{proof}

\begin{corollary}\label{coro}
Let $k=\bar{k}$ an $A/k$ be an abelian variety, then there exists an integral algebraic étale Fourier transform.
\end{corollary}

\begin{proof}
This is a direct consequence of Corollary \ref{coroAlg} and Theorem \ref{teoAlg}.
\end{proof}

\printbibliography[title={Bibliography}]

\info
\end{document}